\documentclass[a4paper,10pt]{amsart}
\usepackage{amssymb}

\numberwithin{equation}{section}
\newtheorem{theorem}{Theorem}[section]
\newtheorem{prop}[theorem]{Proposition}
\newtheorem{lem}[theorem]{Lemma}

\theoremstyle{remark}
\newtheorem{rem}[theorem]{Remark}
\newcommand{\R}{\mathbb{R}}

\newcommand{\N}{\mathbb{N}}

\author[M.~Ltifi]{Maroua Ltifi}
\address{Department of Mathematics, Faculty of Science of Gab\`es, university of Gab\`es; Tunisia}
\email{\sl widaltifi@gmail.com}

\title[Strong solution of the three-dimensional $(3D)$ incompressible magneto-hydrodynamic $(MHD)$ equationss with a modified damping ]
{Strong solution of the three-dimensional $(3D)$ incompressible magneto-hydrodynamic $(MHD)$ equations with modified damping}

\begin{document}
	\begin{abstract}
		This study delves into a comprehensive examination of the three-dimensional $(3D)$ incompressible magneto-hydrodynamic $(MHD)$ equations in $H^{1}(\R^{3})$. The modification involves incorporating a power term in the nonlinear convection component, a particularly relevant adjustment in porous media scenarios, especially when the fluid adheres to the Darcy-Forchheimer law instead of the conventional Darcy law. Our main contributions include establishing global existence over time and demonstrating the uniqueness of solutions. It is important to note that these achievements are obtained with smallness conditions on the initial data, but under the condition that $\beta >3$ and $\alpha>0$
		. However, when $\beta=3$, the problem is limited to the case $0<\alpha<\frac{1}{2}$ as the above inequality is unsolvable for these values of $\alpha$ using our method. To support our statement, we will add a "slight disturbance" of the function f of the type $f(z)=log(e+z^{2})$ or $\log(\log(e^{e}+z^{2}))$ or even $\log(\log(\log((e^{e^{e}})+z^{2})))$.
	\end{abstract}

	

	\maketitle
	\tableofcontents

	
	\section{\bf Introduction}
	
	in This paper, we study the following magnetohydrodynamic system with damping: 
	$$(MHD_{D})\label{sys1}
	\begin{cases}
		\partial_t u
		-\Delta_{h} u-\partial_3^{2}u+ u\nabla u  + b\nabla b+\alpha |u|^{\beta-1}u =\;\;-\nabla p\hbox{ in } \mathbb R^+\times \mathbb R^3\\
		\partial_t b-\Delta_{h} b-\partial_3^{2}b+ b\nabla b  -u\nabla b =\;\;0\\
		{\rm div}\, u = 0, {\rm div}\, b = 0\hbox{ in } \mathbb R^+\times \mathbb R^3\\
		u(0,x) =u^0(x), b(0,x) =b^0(x) \;\;\hbox{ in }\mathbb R^3.
	\end{cases}
	$$
	where $u=u(t,x)=(u_1,u_2,u_3)$, $b=b(t,x)=(b_1,b_2,b_3)$  and $p=p(t,x)$ denote respectively the unknown velocity, the magnetic field and the unknown pressure of the fluid at the point $(t,x)\in \mathbb R^+\times \mathbb R^3$, $\alpha>0$ and $\beta>1$. The terms $v\nabla w:=v_1\partial_1 w+v_2\partial_2 w+v_3\partial_3w$, while $u^0=(u_1^{0}(x),u_2^{0}(x),u_3^{0}(x))$ is an initial given velocity. If $u^0$ is quite regular, the divergence free condition determines the pressure $p$.
	 The damping arises from the resistance to the motion of the flow, describing various physical situations such as flow through porous media, drag or friction effects, and some dissipative mechanisms (see \cite{CJ},\cite{DB1},\cite{DB2},\cite{J},\cite{l}). When $b^{0}=0$ system $(MHD_{D})$ reduces to the Navier-Stokes system with damping.
	It was studied, in the beginning, by Cai and Jiu \cite{CJ} in 2008. They are shown by Galerkin's methods the existence of global weak solution $u\in L^{\infty}(L^{2}(\R^3))\cap L^{2}(\dot{H}^{1}(\R^{3}))\cap L^{\beta+1}(L^{\beta+1}(\R^{3}))$ for $\beta\geq1$, global strong solution for any $\beta\geq\frac{7}{2}$ and that the strong solution is unique for any $\frac{7}{2}\leq\beta\leq5$.\\
	In this sheet we study the large time behaviour in Fourier norms of the solution to the magnetohydrodynamic system with damping in three spatial dimensions in $H^{^1}(\R^{3})$.
		\begin{theorem}\label{the1} For $\alpha>0$ and 	$\beta>3$,
		consider divergence-free vector fields $u^{0}$ and $b^{0}$ $\in H^{1}(\R^{3})$ such that $\|(u^{0},b^{0})\|_{H^{1}}<<\epsilon_{0},$ with $\epsilon_{0}$ is small enough.
		 There exists a global solution $w=(u,b)$ of the magnetohydrodynamics equation $(MHD_{D})$ satisfying the following properties:
		$w=(u,b)\in L^{\infty}(\R^{+},H^{1}(\R^{3}))\cap C(\R^{+},L^{2})\cap L^{2}(\R^{+},\dot{H}^{2}(\R^{3}))\cap L^{\beta+1}(\R^{+},L^{\beta+1}(\R^{3}))$ and $|u|^{\beta-3} |\nabla|u_n|^2|^2$,$|u|^{\beta-1} \nabla |u|^2\in L_{loc}^{1}(\R^{+},L^{1}(\R^{3}))$ we have
		\begin{align}\label{eqth01}
			\|w\|_{L^{2}}^{2}+2\int_{0}^{t}\|\nabla w\|_{L^{2}}^{2}+2\alpha \int_{0}^{t}\|u\|^{\beta+1}_{L^{\beta+1}}\leq\|w^{0}\|^{2}_{L^{2}}.
		\end{align}
		and
		
		\begin{align}\label{eqth02}
			\|\nabla w(t)\|_{L^{2}}^{2}+\int_0^t\|\Delta w\|^2_{L^2}+\alpha\frac{(\beta-1)}{2}\int_0^t\||u|^{\beta-3} |\nabla|u|^2|^2\|_{L^1}+\alpha \int_0^t\||u|^{\beta-1} |\nabla u|^2\|_{L^1}& \leq \|\nabla w^0\|_{L^{2}}^{2}+c_{\alpha,\beta}\|w^0\|_{L^{2}}^{2}\end{align}
		\begin{align}\label{eqth03}
			\|\nabla w(t)\|_{L^{2}}^{2}+\int_0^t\|\Delta w\|^2_{L^2}+\alpha\frac{(\beta-1)}{2}\int_0^t\||u|^{\beta-3} |\nabla|u|^2|^2\|_{L^1}+\alpha \int_0^t\||u|^{\beta-1} |\nabla u|^2\|_{L^1} &\leq \|\nabla w^{0}\|^{2}_{L^{2}}e^{2c_{\alpha,\beta}t}.
		\end{align}
	where, $c_{\alpha,\beta}=\frac{1}{2}\frac{\beta-3}{\beta-1}.\Big(\frac{\alpha(\beta-1)}{2}\Big)^{-\frac{2}{\beta-3}}$and $C$ is a constant that depends on the product law  of the Sobolev spaces.
	\end{theorem}
	\begin{rem}
		To prove this theorem, we use Friederich's method, an interpolation of type $x^{2}\leq c_{\alpha,\beta} +\alpha x^{\beta-1}$ and some weak convergence results in Banach spaces.
	\end{rem}
	Despite this, the continuity and uniqueness of these modified equations remains a big open problem for $\beta=3.$ Indeed, the problem is restricted to the case $0<\alpha<\frac{1}{2}$ since the inequality
	\begin{align*}
			\frac{1}{2}\|\nabla w\|_{L^{2}}^{2}+\|\Delta w\|_{L^{2}}^{2}+\alpha \int_{\R^{3}}|u|^{\beta-1}|\nabla u|^2&\leq\frac{1}{2} \int_{\R^{3}} |u|^{2}|\nabla u|^{2}.	
	\end{align*}
 is not resolvable for these values of $\alpha$ by classical methods and techniques.\\
	The method we employ involves enhancing the function $|u|^2 u$ with a negligible function from a broad class of functions relative to $|u|^\varepsilon$ for $\varepsilon > 0$. This class comprises functions $f:\R^{+}\rightarrow\R^{+}, C^{1}$ that satisfy the hypothesis:
		$$(H)
	\begin{cases}
	\bullet f' >0\\
\bullet	f(0)=0\\
		    \bullet  \forall \beta>3, \exists a_{\beta},b_{\beta}>0 / a_{\beta} z^{2} \leq f(z)\leq b_{\beta} z^{\beta-1}, \;\;\;\forall z\geq1.
	
	\end{cases}
	$$ We consider a more general equation with a damping term, formulated as
	$$(MHD_{f})\label{sys2}
\begin{cases}
	\partial_t u
	-\Delta_{h} u-\partial_3^{2}u+ u\nabla u  + b\nabla b+\alpha f(|u|^{2}) |u|^{2}u =\;\;-\nabla p\hbox{ in } \mathbb R^+\times \mathbb R^3\\
	\partial_t b-\Delta_{h} b-\partial_3^{2}b+ b\nabla u  -u\cdotp\nabla b =\;\;0\\
	{\rm div}\, u = 0, {\rm div}\, b = 0\hbox{ in } \mathbb R^+\times \mathbb R^3\\
	u(0,x) =u^0(x), b(0,x) =b^0(x) \;\;\hbox{ in }\mathbb R^3.
\end{cases}
$$ 
Examples of such functions include $f(z)=\log(e+z)$, $\log(\log(e^{e}+z))$ or even $\log(\log(\log((e^{e})^{e}+z)))$.

Here, it is evident that when \( b^{0} = 0 \), the system reduces to the Navier-Stokes equations with logarithmic damping, where the function \( f(|u|^{2}) = \log(e + |u|^{2}) \) (see \cite{ML}). Furthermore, if \( b^{0} = 0 \) and \( \Delta_{h} = 0 \), the system also reduces to the anistropic Navier-Stokes equations with logarithmic damping (see \cite{MJ}). 
	The main result of our work is illustrated in the following theorem :
	
	\begin{theorem}\label{the2}Consider divergence-free vector fields $u^{0}$ and $b^{0}$ $\in H^{1}(\R^{3})$ such that $\|(u^{0},b^{0})\|_{H^{1}}<<\epsilon_{0},$ with $\epsilon_{0}$ is small enough. Then there exists a global solution $w=(u,b)$ of the magnetohydrodynamics equations $(MHD_{f})$ satisfying the following properties:
		$w=(u,b)\in L^{\infty}(\R^{+},L^{2}(\R^{3}))\cap C(\R^{+},H^{-1})\cap L^{2}(\R+,\dot{H}^{1}(\R^{3}))\cap L^{\beta+1}(\R^{+},L^{\beta+1}(\R^{3}))$and $$f(|u|^{2})|u|^4,\,f'(|u|^{2} |\nabla|u|^2|^2, f(|u|^{2}) |\nabla|u|^2|^2, f(|u|^{2})|u|^2 |\nabla u|^2\in L^1(\R^+,L^1(\R^3)).$$ Moreover, for all $t\geq0$ \begin{equation}\label{eq1}\|w(t)\|_{L^2}^2+2\int_0^t\|\nabla w\|_{L^2}^2+2\alpha\int_0^t\|f(|u|^2)|u|^4\|_{L^1}\leq \|w^0\|_{L^2}^2.\end{equation}
		\begin{align}\label{eq2}
			\nonumber\|\nabla w(t)\|_{L^2}^2+\int_0^t\|\Delta w\|_{L^2}^2+\alpha \int_{0}^{t}\| f'(|u|^{2}) |\nabla|u|^2|^2\|_{L^{1}}+\alpha\int_0^t\| f(|u|^{2}) |\nabla|u|^2|^2\|_{L^1} \\
			+2\alpha\int_0^t\|f(|u|^{2}) |u|^2|\nabla u|^2\|_{L^1}\leq \|\nabla w^0\|_{L^2}^2e^{a_{\alpha}t}\end{align}
		where, $a_{\alpha}=f^{-1}(\frac{1}{2\alpha})$ and $C$ is a constant that depends on the product law  of the Sobolev spaces.
	\end{theorem}
The remainder of our paper is organized as follows. In the second section, we present the notations, definitions, and preliminary results. In Section 3, we examine the global existence and uniqueness of solutions, as established in Theorem \ref{the1} and Theorem \ref{the2}. 
	\section{\bf Notations and preliminary results}
	\subsection{Notations} This section contains some notations and definitions that will be useful later.\\
	\begin{enumerate}
		\item[$\bullet$] The Fourier transformation is normalized as
		$$
		\mathcal{F}(f)(\xi)=\widehat{f}(\xi)=\int_{\mathbb R^3}\exp(-ix.\xi)f(x)dx,\,\,\,\xi=(\xi_1,\xi_2,\xi_3)\in\mathbb R^3.
		$$
		\item[$\bullet$] The inverse Fourier formula is
		$$
		\mathcal{F}^{-1}(g)(x)=(2\pi)^{-3}\int_{\mathbb R^3}\exp(i\xi.x)g(\xi)d\xi,\,\,\,x=(x_1,x_2,x_3)\in\mathbb R^3.
		$$
		\item[$\bullet$] The convolution product of a suitable pair of function $f$ and $g$ on $\mathbb R^3$ is given by
		$$
		(f\ast g)(x):=\int_{\mathbb R^3}f(y)g(x-y)dy.
		$$
		\item[$\bullet$] If $f=(f_1,f_2,f_3)$ and $g=(g_1,g_2,g_3)$ are two vector fields, we set
		$$
		f\otimes g:=(g_1f,g_2f,g_3f),
		$$
		and
		$$
		{\rm div}\,(f\otimes g):=({\rm div}\,(g_1f),{\rm div}\,(g_2f),{\rm div}\,(g_3f)).
		$$
		Moreover, if $\rm{div}\,g=0$ we obtain
		$$
		{\rm div}\,(f\otimes g):=g_1\partial_1f+g_2\partial_2f+g_3\partial_3f:=g.\nabla f.
		$$
		\item[$\bullet$] Let $(B,||.||)$, be a Banach space, $1\leq p \leq\infty$ and  $T>0$. We define $L^p_T(B)$ the space of all
		measurable functions $[0,t]\ni t\mapsto f(t) \in B$ such that $t\mapsto||f(t)||\in L^p([0,T])$.\\
		\item[$\bullet$] The Sobolev space $H^s(\mathbb R^3)=\{f\in \mathcal S'(\mathbb R^3);\;(1+|\xi|^2)^{s/2}\widehat{f}\in L^2(\mathbb R^3)\}$.\\
		\item[$\bullet$] The homogeneous Sobolev space $\dot H^s(\mathbb R^3)=\{f\in \mathcal S'(\mathbb R^3);\;\widehat{f}\in L^1_{loc}\;{\rm and}\;|\xi|^s\widehat{f}\in L^2(\mathbb R^3)\}$.\\
		\item[$\bullet$] For $R>0$, the Friedritch operator $J_R$ is defined by
		$$J_R(D)f=\mathcal F^{-1}({\bf 1}_{|\xi|<R}\widehat{f}).$$
		\item[$\bullet$] The Leray projector $\mathbb P:(L^2(\R^3))^3\rightarrow (L^2(\R^3))^3$ is defined by
		$$\mathcal F(\mathbb P f)=\widehat{f}(\xi)-(\widehat{f}(\xi).\frac{\xi}{|\xi|})\frac{\xi}{|\xi|}=M(\xi)\widehat{f}(\xi);\;M(\xi)=(\delta_{k,l}-\frac{\xi_k\xi_l}{|\xi|^2})_{1\leq k,l\leq 3}.$$
		\item[$\bullet$] $L^2_\sigma(\R^3)=\{f\in (L^2(\R^3))^3;\;{\rm div}\,f=0\}$.
		\item[$\bullet$] $\dot H^1_\sigma(\R^3)=\{f\in (\dot H^1(\R^3))^3;\;{\rm div}\,f=0\}$.
		\item[$\bullet$] $C_{r}(I,B)=\{f:I\rightarrow B\mbox{ right continuous }\}$ , where $B$ is Banach space and $I$ is an interval.
		\item[$\bullet$] Let $a\in\R,$ we define $a_{+}=\max(a,0)$.
	\end{enumerate}
	\subsection{Preliminary results}
	In this section, we recall some classical results and we give new technical lemmas.
	\begin{prop}(\cite{HBAF})\label{prop1} Let $H$ be Hilbert space.
		\begin{enumerate}
			\item If $(x_n)$ is a bounded sequence of elements in $H$, then there is a subsequence $(x_{\varphi(n)})$ such that
			$$(x_{\varphi(n)}|y)\rightarrow (x|y),\;\forall y\in H.$$
			\item If $x\in H$ and $(x_n)$ is a bounded sequence of elements in $H$ such that
			$$(x_n|y)\rightarrow (x|y),\;\forall y\in H.$$
			Then $\|x\|\leq\liminf_{n\rightarrow\infty}\|x_n\|.$
			\item If $x\in H$ and $(x_n)$ is a bounded sequence of elements in $H$ such that
			$$\begin{array}{l}
				(x_n|y)\rightarrow (x|y),\;\forall y\in H\\
				\limsup_{n\rightarrow\infty}\|x_n\|\leq \|x\|,\end{array}$$
			then $\lim_{n\rightarrow\infty}\|x_n-x\|=0.$
		\end{enumerate}
	\end{prop}
	\begin{lem}(\cite{JYC})\label{LP}
		Let $s_1,\ s_2$ be two real numbers and $d\in\N$.
		\begin{enumerate}
			\item If $s_1<d/2$\; and\; $s_1+s_2>0$, there exists a constant  $C_1=C_1(d,s_1,s_2)$, such that: if $f,g\in \dot{H}^{s_1}(\mathbb{R}^d)\cap \dot{H}^{s_2}(\mathbb{R}^d)$, then $f.g \in \dot{H}^{s_1+s_2-1}(\mathbb{R}^d)$ and
			$$\|fg\|_{\dot{H}^{s_1+s_2-\frac{d}{2}}}\leq C_1 (\|f\|_{\dot{H}^{s_1}}\|g\|_{\dot{H}^{s_2}}+\|f\|_{\dot{H}^{s_2}}\|g\|_{\dot{H}^{s_1}}).$$
			\item If $s_1,s_2<d/2$\; and\; $s_1+s_2>0$ there exists a constant $C_2=C_2(d,s_1,s_2)$ such that: if $f \in \dot{H}^{s_1}(\mathbb{R}^d)$\; and\; $g\in\dot{H}^{s_2}(\mathbb{R}^d)$, then  $f.g \in \dot{H}^{s_1+s_2-1}(\mathbb{R}^d)$ and
			$$\|fg\|_{\dot{H}^{s_1+s_2-\frac{d}{2}}}\leq C_2 \|f\|_{\dot{H}^{s_1}}\|g\|_{\dot{H}^{s_2}}.$$
		\end{enumerate}
	\end{lem}
	
	\begin{lem}\cite{ML}\label{lem45}
		Let $A,T>0$ and $f,g,h:[0,T]\rightarrow\R^+$ three continuous functions such that
		\begin{align}\label{e1}\;\;\;\;\forall t\in[0,T];\;f(t)+\int_0^tg(z)dz & \leq A+\int_0^th(z)f(z)dz.\end{align}
		Then $$\forall t\in[0,T];\;f(t)+\int_0^tg(z)dz\leq A\exp(\int_0^th(z)dz).$$
	\end{lem}

	\begin{lem}\cite{ML}\label{lem24}
		Let $\alpha>0$, $\beta>3$ and $x\in\R_{+}$ then
		$$x^{2}\leq 2c_{\alpha,\beta} +\alpha x^{\beta-1},$$
		with $c_{\alpha,\beta}=\frac{1}{2}\frac{\beta-3}{\beta-1}.\Big(\frac{\alpha(\beta-1)}{2}\Big)^{-\frac{2}{\beta-3}}$
	\end{lem}
	
	\begin{lem}\label{lem46}
		Let $d\in\N$. Then, for all $x,y\in\R^d$, we have
		$$\langle f(|x|^{2})|x|^2x-f(|y|^{2})|y|^2y,x-y\rangle\geq0$$	
	\end{lem}
	\begin{proof}
		The proof is a generalization of the lemma in \cite{ML}.\\
		Let $a(z)=f(z^2)z^2$ and suppose that $|x|\geq|y|$ :
		\begin{align*}
			\langle a(|x|)x-a(|y|)y,x-y\rangle=&\langle(a(|x|)-a(|y|))x,x-y\rangle+a(|y|)\langle x-y,x-y\rangle\\=&(a(|x|)-a(|y|))\langle x,x-y\rangle+a(|y|)|x-y|^{2}.
		\end{align*}	
	$f$  is a strictly increasing positive function and by using $|x|\geq|y|$  , we get
		$$\langle x,x-y\rangle=|x|^2-\langle x,y\rangle\geq|x|^2-|x||y|=|x|(|x|-|y|)\geq0,$$
		this yields the desired result.	
	\end{proof}

	\section{\bf Existence and uniqueness of strong solution }
	\subsection{Proof of Theorem $\ref{the1}$}
	$\bullet$ {\bf{A priori estimates:}} We initiate our analysis by seeking an $L^2(\R^{3})$ uniform estimate for the velocity. To achieve this, we begin by multiplying the first equation of the $(MHD_{D})$ system by $w$ and integrating over $\mathbb{R}^{3}$. Subsequently, we integrate with respect to time, resulting in the inequality: 
	\begin{align}\label{eq01}\|w(t)\|_{L^2}^2+2\int_0^t\|\nabla w\|_{L^2}^2+2\alpha\int_0^t \|u\|_{L^{\beta+1}}^{\beta+1} \leq \|w^0\|_{L^2}^2.\end{align}
	Furthermore, we proceed by taking the $\dot{H}^{1}(\R^{3})$ scalar product with $w$, leading to:
	\begin{align}\label{eqh1}
		\frac{1}{2}\frac{d}{dt}\|\nabla w\|_{L^{2}}^{2}+\|\Delta w\|^2_{L^2}+ \int_{\R^3} \nabla(|u|^{\beta-1}u) \nabla u &\leq \sum_{i=1}^{4} |I_{k}| .
	\end{align}
where $$ |I_{1}|=| \langle \nabla (u \nabla u),\nabla u\rangle_{L^{2}}|, |I_{2}|=| \langle \nabla (b \nabla b),\nabla u\rangle_{L^{2}}|, I_{3}|=| \langle \nabla (b \nabla u),\nabla b\rangle_{L^{2}}| \;\textit{and}\; I_{4}|=| \langle \nabla (u \nabla b),\nabla b\rangle_{L^{2}}|$$
	Utilizing the identity
	$$\partial_j\Big(|u|^{\beta-1}u\Big)\partial_ju=|u|^{\beta-1}|\partial_ju|^2+\frac{(\beta-1)}{4}|u|^{\beta-3}|\partial_j|u|^2|^2$$
	yields
	$$\int_{\R^3} \nabla(|u|^{\beta-1}u) \nabla u=\||u|^{\beta-1}|\nabla u|^2\|_{L^1}+\frac{(\beta-1)}{4}\||u|^{\beta-3}|\nabla|u|^2|^2\|_{L^1}.$$
Furthermore, due to ${\rm div} (u) = 0$ and ${\rm div} (b) = 0$, we have
	$$|I_{1}|=|\langle \nabla (u \nabla u),\nabla u\rangle_{L^{2}}|=|\langle u \nabla u,\Delta  u\rangle_{L^{2}}|$$
$$|I_{2}|=|\langle \nabla (b\nabla b),\nabla u \rangle_{L^{2}}|=|\langle  b\nabla b, \Delta u\rangle_{L^{2}}|$$
$$|I_{3}|=|\langle \nabla (b\nabla u),\nabla b\rangle_{L^{2}}|=|\langle b\nabla u,\Delta b\rangle_{L^{2}}|$$
and $$|I_{4}|=|\langle \nabla (u\nabla b),\nabla b\rangle_{L^{2}}|=|\langle \nabla u\nabla b,\Delta b\rangle_{L^{2}}|$$
Thus,
\begin{align*}
	|I_{1}|&\leq \|u\nabla u\|_{L^{2}}\|\Delta u\|_{L^{2}}\\&\leq \frac{1}{2}\|u\nabla u\|^{2}_{L^{2}}+\frac{1}{2}\|\Delta u\|^{2}_{L^{2}}
\end{align*}
By applying Lemma \ref{lem24}, we obtain
\begin{align*}
	\|u \nabla u\|_{L^2}^2\leq 2c_{\alpha,\beta}\|\nabla u\|_{L^2}^2+\alpha\||u|^{\beta-1} |\nabla u|^2\|_{L^1}\end{align*}
which implies 
\begin{align*}|I_{1}|&\leq c_{\alpha,\beta}\|\nabla u\|_{L^2}^2+\frac{\alpha}{2}\||u|^{\beta-1} |\nabla u|^2\|_{L^1}+\frac{1}{2}\|\Delta  u\|_{L^{2}}^2\\&\leq c_{\alpha,\beta}\|\nabla w\|_{L^2}^2+\frac{\alpha}{2}\||u|^{\beta-1} |\nabla u|^2\|_{L^1}+\frac{1}{2}\|\Delta  u\|_{L^{2}}^2 \end{align*}
	
and
\begin{align*}
	|I_{2}|&\leq \| b\nabla b\|_{L^{2}}\|\Delta u\|_{L^{2}}\\&\leq \|b\|_{\dot{H}^{1}}\|\nabla b\|_{\dot{H}^{\frac{1}{2}}}\|\Delta u\|_{L^{2}}
\end{align*} Interpolating, we find
\begin{align*}
	|I_{2}|&\leq C \|b\|_{\dot{H}^{1}}\|\nabla b\|^{\frac{1}{2}}_{L^{2}} \|\nabla b\|^{\frac{1}{2}}_{\dot{H}^{1}}\|\Delta u\|_{L^{2}},
\end{align*}
and applying interpolation again,
\begin{align*}
	|I_{2}|&\leq C \|\nabla b\|_{L^{2}}\|\nabla b\|^{\frac{1}{2}}_{L^{2}} \|\Delta b\|^{\frac{1}{2}}_{L^{2}}\|\Delta u\|_{L^{2}}\\&\leq C \|\nabla w\|_{L^{2}}\|\nabla w\|^{\frac{1}{2}}_{L^{2}} \|\Delta w\|^{\frac{1}{2}}_{L^{2}}\|\Delta w\|_{L^{2}}
\end{align*}
which, by further interpolation, yields
\begin{align*}
	|I_{2}|&\leq C \|w\|^{1/2}_{L^{2}}\|\Delta w\|^{1/2}_{L^{2}}\| \nabla w\|^{\frac{1}{2}}_{L^{2}}  \|\Delta w\|^{\frac{1}{2}}_{L^{2}}\|\Delta w\|_{L^{2}}\\
	&\leq C\|w\|^{1/2}_{L^{2}}\| \nabla w\|^{\frac{1}{2}}_{L^{2}}\|\Delta w\|^{2}_{L^{2}}\\
	&\leq C\| w\|_{H^{1}}\|\Delta w\|^{2}_{L^{2}}
\end{align*}
Following the same procedure for both $I_{3}$ and $I_{4}$, we arrive at:\\

By incorporating these inequalities into (\ref{eqh1}), we obtain:
	\begin{align*}
		\frac{1}{2}\frac{d}{dt}\|\nabla w\|_{L^{2}}^{2}+\|\Delta w\|_{L^{2}}^{2}+\alpha \||u|^{\beta-1}|\nabla u|^2\|_{L^1}\\+\alpha \frac{(\beta-1)}{4}\||u|^{\beta-3}|\nabla|u|^2|^2\|_{L^1}&\leq 3C\| w\|_{H^{1}}\|\Delta w\|^{2}_{L^{2}}\\&+c_{\alpha,\beta}\|\nabla w\|_{L^2}^2+\frac{\alpha}{2}\||u|^{\beta-1} |\nabla u|^2\|_{L^1}+\frac{1}{2}\|\Delta  w\|_{L^{2}}^2 \\
		&\leq  (3C\| w\|_{H^{1}}+\frac{1}{2})\|\Delta w\|^{2}_{L^{2}}+c_{\alpha,\beta}\|\nabla w\|_{L^2}^2.
	\end{align*}
	
Let $C_{0}$ belong to the interval $(0,\frac{1}{12C})$ consequently,

	\begin{align*}
		 3C\| w^{0}\|_{H^{1}}\leq \frac{1}{2}\Longleftrightarrow  \| w^{0}\|_{H^{1}}\leq \frac{1}{6C}. 
	\end{align*}
	
	Due to the continuity of the function $(t\longmapsto\|w(t)\|_{H^{1}})$, we obtain
	\begin{align*}
		T&=\sup\{t\geq0/ \|w\|_{L^{\infty}([0,t],H^{1})}<\frac{1}{2}(\| w^{0}\|_{H^{1}}+\frac{1}{6C})\}\in (0,\infty]
	\end{align*}
	Since $\frac{1}{2}(\| w^{0}\|_{H^{1}}+\frac{1}{6C}) $ lies within $(\| w^{0}\|_{H^{1}},\frac{1}{6C})$, then for $t\in [0,T)$, we have
	\begin{align*}
	\frac{1}{2}\frac{d}{dt}\|\nabla w\|_{L^{2}}^{2}+\|\Delta w\|_{L^{2}}^{2}+\alpha \||u|^{\beta-1}|\nabla u|^2\|_{L^1}\\+\alpha \frac{(\beta-1)}{4}\||u|^{\beta-3}|\nabla|u|^2|^2\|_{L^1}&\leq 3C  \|w\|_{H^{1}}\|\Delta w_{n}\|^{2}_{L^{2}}\\&+c_{\alpha,\beta}\|\nabla w\|_{L^2}^2+\frac{\alpha}{2}\||u|^{\beta-1} |\nabla u|^2\|_{L^1}+\frac{1}{2}\|\Delta  w\|_{L^{2}}^2 \\
		\frac{1}{2}\frac{d}{dt}\|\nabla w\|_{L^{2}}^{2}+\|\Delta w\|_{L^{2}}^{2}+\alpha \frac{(\beta-1)}{4}\||u|^{\beta-3}|\nabla|u|^2|^2\|_{L^1}&\leq  (3C \big(\frac{1}{2}(\| w^{0}\|_{H^{1}}+\frac{1}{6C})+\frac{1}{2}\big))\|\Delta w\|^{2}_{L^{2}}+c_{\alpha,\beta}\|\nabla w\|_{L^2}^2\\&\leq    \big(\frac{3}{2}C\| w^{0}\|_{H^{1}}+\frac{1}{4}+\frac{1}{2}\big))\| \Delta w\|^{2}_{L^{2}}+c_{\alpha,\beta}\|\nabla w\|_{L^2}^2\\
		&\leq  \big(\frac{3}{2}C\| w^{0}\|_{H^{1}}+\frac{3}{4}\big))\|\Delta w\|^{2}_{L^{2}}+c_{\alpha,\beta}\|\nabla w\|_{L^2}^2.
\end{align*}
Then
		\begin{align*}
		\frac{1}{2}\frac{d}{dt}\|\nabla w\|_{L^{2}}^{2}+(\frac{1}{4}-\frac{3}{2}C\|w^{0}\|_{H^{1}})\|\Delta w\|_{L^{2}}^{2}+\frac{\alpha}{2} \||u|^{\beta-1}|\nabla u|^2\|_{L^1}\\+\alpha \frac{(\beta-1)}{4}\||u|^{\beta-3}|\nabla|u|^2|^2\|_{L^1}&\leq c_{\alpha,\beta}\|\nabla w\|_{L^2}^2.
	\end{align*}
	Consequently, by using \ref{eq01} and for $t\in [0,T)$, we obtain
	
	\begin{align}
		\|\nabla w\|_{L^{2}}^{2}+2(\frac{1}{4}-\frac{3}{2}C\| w^{0}\|_{H^{1}})\int_{0}^{t}\|\Delta w\|_{L^{2}}^{2}\\\nonumber+\alpha \int_{0}^{t} \||u|^{\beta-1}|\nabla u|^2\|_{L^1}+\alpha \frac{(\beta-1)}{2}\int_{0}^{t}\||u|^{\beta-3}|\nabla|u|^2|^2\|_{L^1}&\leq\|\nabla w^{0}\|^{2}_{L^{2}}+2c_{\alpha,\beta}\int_{0}^{t} \|\nabla w\|_{L^2}^2\\&\leq\|\nabla w^{0}\|^{2}_{L^{2}}+c_{\alpha,\beta} \|w^{0}\|_{L^2}^2
	\end{align}
	
	This implies $t=\infty$ under the condition $\frac{3}{2}C\| w^{0}\|_{H^{1}}<\frac{1}{4}$, we get $\forall t\geq 0$ :
	\begin{align}\label{p2eq01}
		\|\nabla w\|_{L^{2}}^{2}+\int_{0}^{t}\|\Delta w\|_{L^{2}}^{2}+\alpha \int_{0}^{t} \||u|^{\beta-1}|\nabla u|^2\|_{L^1}+\alpha \frac{(\beta-1)}{2}\int_{0}^{t}\||u|^{\beta-3}|\nabla|u|^2|^2\|_{L^1}&\leq\|\nabla w^{0}\|^{2}_{L^{2}}+c_{\alpha,\beta} \| w^{0}\|_{L^2}^2.
	\end{align}

	By applying Lemma \ref{lem45} to inequality \ref{eq01}, we get
	$$\|\nabla w(t)\|_{L^{2}}^{2}+\int_0^t\|\Delta w\|^2_{L^2}+\alpha\frac{(\beta-1)}{2}\int_0^t\||u|^{\beta-3} |\nabla|u|^2|^2\|_{L^1}+\alpha \int_0^t\||u|^{\beta-1} |\nabla u|^2\|_{L^1} \leq \|\nabla w^{0}\|^{2}_{L^{2}}e^{2c_{\alpha,\beta}t}.$$
	
	we obtain the global existence for bounded solution.\\
	$\bullet$ {\bf Passage to the limit:} Definitely, these bounds come from the approximate solutions via the Friederich's regularization procedure. The transition to the limit follows using classical argument by combining Ascoli's Theorem and the Cantor Diagonal Process \cite{HB}. Inequalities (\ref{eqth01})-(\ref{eqth02})-(\ref{eqth03}) are given by the above inequalities. And this solution $u$ is in $L^{\infty}(\R^{+},H^1)\cap L^2(\R^{+},\dot{H^{2}}) \cap C(\R^{+},L^{2})$.\\
	$\bullet$ {\bf Uniqueness:} The uniqueness is given by energy method in $L^2$, which ends the proof of Theorem \ref{the1}.\\
	\begin{rem}
		For $\beta=3$ Indeed, the problem is limited to the case $0<\alpha<\frac{1}{2}$ because the inequality (\ref{eqh1}) is unsolvable for these $\alpha$ values.
		To solve our statement, we will add the function  $f(|u|^{2})$ to $|u|^{2}u$.
		We will solve  the magnetohydrodynamic equation with damping by a light function $f$ $(MHD_{f})$ at the next subsection.
	\end{rem}
	\subsection{Proof of Theorem $\ref{the2}$}
	$\bullet$ {\bf{A priori estimates}}  \\
We initiate our analysis by seeking an $L^2(\R^{3})$ uniform estimate for the velocity. To achieve this, we begin by multiplying the first equation of the $(MHD_{f})$ system by $w$ and integrating over $\mathbb{R}^{3}$. Subsequently, we integrate with respect to time, resulting in the inequality: 
	\begin{align}\label{eq1}\|w(t)\|_{L^2}^2+2\int_0^t\|\nabla w\|_{L^2}^2+2\alpha\int_0^t\|f(|u|^{2})|u|^4\|_{L^1}\leq \|w^0\|_{L^2}^2.\end{align}
Furthermore, we proceed by taking the $\dot{H}^{1}(\R^{3})$ scalar product with $w$, leading to:
	\begin{align*}
		\frac{1}{2}\frac{d}{dt}\|\nabla w\|^2+\frac{1}{2}\|\Delta w\|^2_{L^2}+\frac{\alpha}{2} \int_{\R^3}  f'(|u^{2}|) |\nabla|u|^2|^2+\frac{\alpha}{2} \int_{\R^3} f(|u^{2}|) |\nabla|u|^2|^2&\\+ \int_{\R^3}\alpha f(|u^{2}|)|u|^2 |\nabla u|^2 &\leq \sum_{i=1}^{4} |J_{i}|.
	\end{align*}
Where $$ |J_{1}|=| \langle \nabla (u \nabla u),\nabla u\rangle_{L^{2}}|, |J_{2}|=| \langle \nabla (b \nabla b),\nabla u\rangle_{L^{2}}|, |J_{3}|=| \langle \nabla (b \nabla u),\nabla b\rangle_{L^{2}}| \;\textit{and}\; |J_{4}|=| \langle \nabla (u \nabla b),\nabla b\rangle_{L^{2}}|.$$
Using the fact that $$|J_{1}|=|\langle \nabla (u \nabla u),\nabla u\rangle_{L^{2}}|\leq\frac{1}{2}\|u \nabla u\|_{L^2}^2+\frac{1}{2}\|\Delta  u\|_{L^{2}}^2$$ and 
\begin{align*}
	|J_{2}|+|J_{3}|+|J_{4}|&\leq 3 C \|w\|_{H^{1}}\|\Delta w\|^{1/2}_{L^{2}}\| \nabla w\|^{\frac{1}{2}}_{L^{2}}  \|\Delta w\|^{\frac{1}{2}}_{L^{2}}\|\Delta w\|_{L^{2}}\\
	&\leq C\|w\|^{1/2}_{L^{2}}\| \nabla w\|^{\frac{1}{2}}_{L^{2}}\|\Delta w\|^{2}_{L^{2}}\\
	&\leq C\| w\|_{H^{1}}\|\Delta w\|^{2}_{L^{2}}
\end{align*}
	\begin{align*}
			\frac{1}{2}\frac{d}{dt}\|\nabla w\|^2+\frac{1}{2}\|\Delta w\|^2_{L^2}+\frac{\alpha}{2} \int_{\R^3}  f'(|u^{2}|) |\nabla|u|^2|^2&\\+\frac{\alpha}{2} \int_{\R^3} f(|u^{2}|) |\nabla|u|^2|^2+ \int_{\R^3}\alpha f(|u^{2}|)|u|^2 |\nabla u|^2&\leq 3C  \| w\|_{H^{1}}\|\Delta w_{n}\|^{2}_{L^{2}}+\frac{1}{2}\|\Delta  w\|_{L^{2}}^2 \\
		&\leq  (3C \| w\|_{H^{1}}+\frac{1}{2})\|\Delta w\|^{2}_{L^{2}}+\frac{1}{2}\|u\cdot\nabla u\|^{2}_{L^{2}}.
	\end{align*}
		
	Let $C_{0}$ belong to the interval $(0,\frac{1}{12C})$ consequently,

	\begin{align*}
		3C\| w^{0}\|_{H^{1}}\leq \frac{1}{2}\Longleftrightarrow \| w^{0}\|_{H^{1}}\leq \frac{1}{6C}. 
	\end{align*}
	
	Due to the continuity of the function $(t\longmapsto\|w(t)\|_{H^{1}})$, we obtain
	\begin{align*}
		T&=\sup\{t\geq0/ \|w_{n}\|_{L^{\infty}([0,t],H^{1})}<\frac{1}{2}(\| w^{0}\|_{H^{1}}+\frac{1}{6C})\}\in (0,\infty]
	\end{align*}
	Since $\frac{1}{2}(\| w^{0}\|_{H^{1}}+\frac{1}{6C}) $ lies within $(\| w^{0}\|_{H^{1}},\frac{1}{6C})$,

\begin{align*}
	\frac{1}{2}\frac{d}{dt}\|\nabla w\|^2+\frac{1}{2}\|\Delta w\|^2_{L^2}+\frac{\alpha}{2} \int_{\R^3}  f'(|u^{2}|) |\nabla|u|^2|^2&\\+\frac{\alpha}{2} \int_{\R^3} f(|u^{2}|) |\nabla|u|^2|^2+ \int_{\R^3}\alpha f(|u^{2}|)|u|^2 |\nabla u|^2&\leq \big(3C\frac{1}{2}(\| w^{0}\|_{H^{1}}+\frac{1}{6C})+\frac{1}{2}\big)\|\Delta w\|_{L^{2}}^{2}+\frac{1}{2}\|u\nabla u\|^{2}_{L^{2}},
\end{align*}
then
\begin{align*}
	\frac{1}{2}\frac{d}{dt}\|\nabla w\|_{L^{2}}^{2}+(\frac{1}{4}-\frac{3}{2}\| w^{0}\|_{H^{1}})\|\Delta w\|_{L^{2}}^{2}+\frac{\alpha}{2} \int_{\R^3}  f'(|u^{2}|) |\nabla|u|^2|^2&\\+\frac{\alpha}{2} \int_{\R^3} f(|u^{2}|) |\nabla|u|^2|^2+ \int_{\R^3}(\alpha f(|u^{2}|)-\frac{1}{2})|u|^2 |\nabla u|^2&\leq0.
\end{align*}
	To continue the study of our system, we need to discuss according to the position of $\alpha$ with respect to $1/2$:\\
Now, suppose that  for $t\geq0$, put $$M_{t}=\{x\in\R^3:~~\alpha f(|u|^2)-\frac{1}{2}\geq0\}.$$
Clearly, we have
$$x\notin M_{t}\Longleftrightarrow |u(t,x)|^2<  f^{-1}(\frac{1}{2\alpha}).$$
Further
\begin{align*}
	\int_{\R^3}(\frac{1}{2}-\alpha f(|u|^2) |\nabla u|^2&=\int_{M_{t}}(\frac{1}{2}-\alpha f(|u|^2)|u|^2 |\nabla u|^2\\&+\int_{M^c_{t}}(\frac{1}{2}-\alpha f(|u|^2)|u|^2 |\nabla u|^2\\
	&\leq \int_{M^c_{t}}(\frac{1}{2}-\alpha f(|u|^2)|u|^2 |\nabla u|^2\\
	&\leq \frac{1}{2}\int_{M^c_{t}}|u|^2 |\nabla u|^2\\
	&\leq \frac{1}{2}(f^{-1}(\frac{1}{2\alpha}))\int_{M^c_{t}}|\nabla u|^2.
\end{align*}
Since
\begin{align*}
	\int_{M^c_{t}}(\frac{1}{2}-\alpha f(|u|^2))|u|^2 |\nabla u|^2
	&\leq \frac{1}{2}(f^{-1}(\frac{1}{2\alpha}))\|\nabla u\|_{L^2}^2.	
\end{align*}
So, in all cases we have
\begin{align*}\frac{1}{2}\frac{d}{dt}\|\nabla u\|^{2}_{L^{2}}+\frac{1}{2}\|\Delta  u\|^{2}_{L^{2}}+\frac{\alpha}{2} \| f'(|u^{2}|) |\nabla|u|^2|^2\|_{L^{1}}\\+\frac{\alpha}{2} \|f(|u^{2}|) |\nabla|u|^2|^2\|_{L^{1}}&\leq \frac{1}{2}(f^{-1}(\frac{1}{2\alpha}))\|\nabla u\|_{L^2}^2.\end{align*}

Consequently, for $t\in [0,T)$, we obtain
\begin{align}
\|\nabla w\|_{L^{2}}^{2}+2(\frac{1}{4}-\frac{3}{2}C\|\nabla w^{0}\|_{L^{2}})\int_{0}^{t}\|\Delta w\|_{L^{2}}^{2}+\alpha
	 \int_{\R^3} \nabla f(|u^{2}|) |\nabla|u|^2|^2&\leq\|\nabla w^{0}\|^{2}_{L^{2}}+(f^{-1}(\frac{1}{2\alpha}))\int_{0}^{t}\|\nabla u\|_{L^2}^2\\ \label{eq02}&\leq\|\nabla w^{0}\|^{2}_{L^{2}}+f^{-1}(\frac{1}{2\alpha})\int_{0}^{t}\|\nabla w\|_{L^2}^2\\\nonumber&\leq\|\nabla w^{0}\|^{2}_{L^{2}}+f^{-1}(\frac{1}{2\alpha})\| w\|_{L^2}^2.
\end{align}

This implies $t=\infty$ under the condition $\frac{3}{2}C\|\nabla w^{0}\|^{2}_{L^{2}}<\frac{1}{4}$, we get : $\forall t\geq0$
\begin{align}
	\|\nabla w\|_{L^{2}}^{2}+\int_{0}^{t}\|\Delta w\|_{L^{2}}^{2}+\alpha \int_{\R^3} \nabla f(|u^{2}|) |\nabla|u|^2|^2+\alpha \int_{\R^3} f(|u^{2}|) |\nabla|u|^2|^2&\leq\|\nabla w^{0}\|^{2}_{L^{2}}+(f^{-1}(\frac{1}{2\alpha}))\|\nabla w\|_{L^2}^2.
\end{align}

	By applying  Lemma \ref{lem45} to \ref{eq02}, we get
	\begin{align}\label{eq20}
	\|\nabla w\|_{L^{2}}^{2}+\int_{0}^{t}\|\Delta w\|_{L^{2}}^{2}+\alpha \int_{\R^3} \nabla f(|u^{2}|) |\nabla|u|^2|^2+\alpha \int_{\R^3} f(|u^{2}|) |\nabla|u|^2|^2&\leq\|\nabla w^{0}\|^{2}_{L^{2}}e^{a_{\alpha}t} \;\; \forall t\geq0,
	\end{align}
	where, $a_{\alpha}=f^{-1}(\frac{1}{2\alpha}).$
	Definitely, these bounds come from the approximate solutions via the Friederich's regularization procedure. 
	The transition to the limit follows using classical argument by combining Ascoli's Theorem and the Cantor Diagonal Process \cite{HB}. And this solution in $L^{\infty}(\R^{+},H^1)\cap L^2(\R^{+},\dot{H^{2}})$ satisfies (\ref{eq1}) and (\ref{eq2}).\\
	$\bullet${\bf{Uniqueness :}}
	The proof is similar to the one presented in \cite{RAJ}.\\
	Let $w=(u,b)$ and $v=(s,y)$ be two solutions of system $(MHD_{f})$.  We take the difference of the corresponding system, we denote $m=u-s$, $n=b-y$, where $p$ is the pressure term corresponding to $w$ and $q$ is the one corresponding to $v$. Thus, we get for $(t,x)\in \R^{+}\times\R^{3}$,
		$$
	\begin{cases}
		\partial_t m
		-\Delta m+ m.\nabla u  + s\nabla b+y\nabla n+\alpha (f(|u|^{2}) |u|^{2}u -f(|s|^{2}) |s|^{2}s)=\;\;-\nabla (p-q)\\
		\partial_t n-\Delta n+ n.\nabla u +y\nabla m-m\nabla b-s\nabla n =\;\;0\\
		{\rm div}\, m = 0, {\rm div}\, n = 0\\
		m(0,x) =u^0(x), n(0,x) =b^0(x) .
	\end{cases}$$

	Taking the $L^2(\R^{3})$ scalar product of the first equation with $m$ and the $L^2(\R^{3})$ scalar product of the second equation with $n$ , yielding  :
	\begin{align*}
		\frac{1}{2}\frac{d}{dt}(\|m\|^{2}_{L^2}+\|n\|^{2}_{L^2})+\|\nabla m\|^{2}_{L^2}+\|\nabla n\|^{2}_{L^2} +\alpha\langle(f(|u|^2)|u|^2u-f(|s|^2)|s|^2s),m\rangle_{L^{2}}+\langle m\nabla u, m\rangle_{L^2}\\+\langle s \nabla m,m\rangle_{L^{2}}+\langle n\nabla b,m\rangle_{L^{2}}+\langle y\nabla n,m\rangle_{L^{2}}+\langle n \nabla u,n \rangle_{L^{2}}+\langle y \nabla m,n\rangle_{L^{2}}+\langle m \nabla b , n\rangle_{L^{2}}+\langle s\nabla n,n\rangle_{L^{2}}&=0
	\end{align*}

As ${\rm div}\, m = 0, {\rm div}\, n = 0$, we have 
$$\langle s\nabla m,m\rangle_{L^{2}}=0,\langle s\nabla n,n\rangle_{L^{2}}=0 \;\;\textit{and}\;\; \langle\nabla(p-q),m \rangle_{L^{2}}=0$$
also since 
$$\langle y.\nabla n,m\rangle_{L^{2}}+\langle y.\nabla m,n\rangle_{L^{2}}=\langle y.\nabla (m+n),(nm+)\rangle_{L^{2}}-\langle y.\nabla m,m\rangle_{L^{2}}-\langle y.\nabla n,n\rangle_{L^{2}}$$
it vanisges thanks to the divergence free condition and 
	Thanks to the Lemma \ref{lem46}, we get:
	\begin{align*}
		\frac{1}{2}\frac{d}{dt}(\|m\|^{2}_{L^2}+\|n\|^{2}_{L^2})+\|\nabla m\|^{2}_{L^2}+\|\nabla n\|^{2}_{L^2}&\leq |\langle s .\nabla m,m\rangle_{L^{2}}|+|\langle n.\nabla b,m\rangle_{L^{2}}|+|\langle n .\nabla u,n \rangle_{L^{2}}|+|\langle m .\nabla b , n\rangle_{L^{2}}|\\&\leq K_{1}+K_{2}+K_{3}+k_{4}.
	\end{align*}
Since 
\begin{align*}
K_{2}=	\langle n.\nabla b,m\rangle_{L^{2}}&\leq \|nb\|_{L^{2}}\|\nabla m\|_{L^{2}}\\&\leq \|n\|_{L^{3}} \|b\|_{L^{6}}\|\nabla m\|_{L^{2}}\\& \leq c \|n\|_{L^{3}} \|b\|_{\dot{H}^{1}}\|\nabla m\|_{L^{2}}\\& \leq c \|n\|_{\dot{H}^{\frac{1}{2}}} \|b\|_{\dot{H}^{1}}\|\nabla m\|_{L^{2}}\\& \leq  c\|n\|^{\frac{1}{2}}_{\dot{H}^{0}}\|n\|^{\frac{1}{2}}_{\dot{H}^{1}} \|b\|_{\dot{H}^{1}}\|\nabla m\|_{L^{2}}
\end{align*}
as $\|b\|_{\dot{H}^{1}}\leq C_{0}=\|a^{0},b^{0}\|_{\dot{H}^{1}}$
	we get 
	\begin{align*}
		K_{2}=\langle n.\nabla b,m\rangle_{L^{2}}& \leq C_{0} c\|n\|^{\frac{1}{2}}_{\dot{H}^{0}}\|n\|^{\frac{1}{2}}_{\dot{H}^{1}}\|\nabla m\|_{L^{2}}
	\end{align*}
Interpolation inequality leads to 
 \begin{align*}
 	K_{2}& \leq C_{0} \|(n,m)\|^{\frac{1}{2}}_{L^{2}}  \|(n,m)\|^{\frac{3}{2}}_{\dot{H}^{1}}
 \end{align*}
 Yong inequality gives 
  \begin{align*}
 	K_{2}& \leq C'_{0} \|(n,m)\|^{2}_{L^{2}} +\frac{1}{4} \|(n,m)\|^{2}_{\dot{H}^{1}}
 \end{align*}
Following the same procedure for $K_{1}$,  $K_{3}$ and $K_{4}$, we arrive at:
\begin{align*}
	\frac{1}{2}\frac{d}{dt}(\|m\|^{2}_{L^2}+\|n\|^{2}_{L^2})+\|\nabla m\|^{2}_{L^2}+\|\nabla n\|^{2}_{L^2}&\leq 4C_{0}\|(n,m)\|^{2}_{L^{2}} .
\end{align*}
	According to Gronwall Lemma , we obtain :
	$$\|(n,m)\|^{2}_{L^{2}}\leq \|(n,m)(0)\|^2_{L^{2}}e^{Ct},$$	
	but $(n,m)(0)=0$, so $n=m$.\\
\section*{Declarations}

\textbf{Competing Interests}
The author declares no competing interests.

\end{document}